\documentclass{amsart}
\usepackage[utf8x]{inputenc}
\usepackage{amsmath, amsthm, amssymb}
\usepackage[usenames,dvipsnames,svgnames,table]{xcolor}
\usepackage[margin=1.3in]{geometry}

\newtheorem{theorem}{Theorem}[section]
\newtheorem{proposition}[theorem]{Proposition}
\newtheorem{lemma}[theorem]{Lemma}
\newtheorem*{remark}{Remark}

\newcommand{\R}{\mathbb R}
\newcommand{\eps}{\varepsilon}

\newcommand{\dd}{\, \mathrm{d}}

\numberwithin{equation}{section}

\newcommand{\commentout}[1]{}

\title[Global estimates for the inhomogeneous Landau equation]{Global a priori estimates for the inhomogeneous Landau equation with moderately soft potentials}
\author{Stephen Cameron}
\address{Department of Mathematics, University of Chicago, 5734 S. University Ave., Chicago, IL 60637}
\email{scameron@math.uchicago.edu}
\author{Luis Silvestre}
\address{Department of Mathematics, University of Chicago, 5734 S. University Ave., Chicago, IL 60637}
\email{luis@math.uchicago.edu}
\author{Stanley Snelson}
\address{Department of Mathematics, University of Chicago, 5734 S. University Ave., Chicago, IL 60637}
\email{snelson@math.uchicago.edu}

\thanks{LS was partially supported by NSF grants DMS-1254332 and DMS-1362525. SS was partially supported by NSF grant DMS-1246999.}
\begin{document}

\begin{abstract}
We establish \emph{a priori} upper bounds for solutions to the spatially inhomogeneous Landau equation in the case of moderately soft potentials, with arbitrary initial data, under the assumption that mass, energy and entropy densities stay under control. Our pointwise estimates decay polynomially in the velocity variable. We also show that if the initial data satisfies a Gaussian upper bound, this bound is propagated for all positive times.
\end{abstract}
\maketitle

\section{Introduction}

We consider the spatially inhomogeneous Landau equation, a kinetic model from plasma physics that describes the evolution of a particle density $f(t,x,v)\geq 0$ in phase space (see, for example, \cite{chapmancowling, lifschitzpitaevskii}). It is written in divergence form as
\begin{equation}\label{e:divergence}
\partial_t f + v\cdot \nabla_x f = \nabla_v\cdot \left[\bar a(t,x,v)\nabla_v f\right] + \bar b(t,x,v)\cdot\nabla_v f + \bar c(t,x,v) f,
\end{equation}
where $t\in [0,T_0]$, $x\in\R^d$, and $v\in\R^d$. The coefficients $\bar a(t,x,v)\in \R^{d\times d}$, $\bar b(t,x,v) \in \R^d$, and $\bar c(t,x,v)\in \R$ are given by
\begin{align}
\bar a(t,x,v) &:= a_{d,\gamma}\int_{\R^d} \left( I - \frac w {|w|} \otimes \frac w {|w|}\right) |w|^{\gamma + 2} f(t,x,v-w) \dd w,\label{e:a}\\
\bar b(t,x,v) &:= b_{d,\gamma}\int_{\R^d} |w|^\gamma w f(t,x,v-w)\dd w,\label{e:b}\\
\bar c(t,x,v) &:= c_{d,\gamma}\int_{\R^d} |w|^\gamma f(t,x,v-w)\dd w, \label{e:c}
\end{align}
where $\gamma$ is a parameter in $[-d,\infty)$, and $a_{d,\gamma}$, $b_{d,\gamma}$, and $c_{d,\gamma}$ are constants. When $\gamma = -d$, the formula for $\bar c$ must be replaced by $\bar c = c_{d,\gamma} f$. Equation \eqref{e:divergence} arises as the limit of the Boltzmann equation as grazing collisions predominate, i.e. as the angular singularity approaches 2 (see the discussion in \cite{alexandre2004landau}). The case $d=3, \gamma=-3$, corresponds to particles interacting by Coulomb potentials in small scales. The case $\gamma\in [-d,0)$ is known as \emph{soft potentials}, $\gamma = 0$ is known as \emph{Maxwell molecules}, and $\gamma >0$ \emph{hard potentials}. In this paper, we focus on \emph{moderately soft potentials}, which is the case $\gamma \in (-2,0)$.

We assume that the mass density, energy density, and entropy density are bounded above, and the mass density is bounded below, uniformly in $t$ and $x$:
\begin{align}
0<m_0\leq \int_{\R^d} f(t,x,v)\dd v \leq M_0,\label{e:M0}\\
\int_{\R^d} |v|^2 f(t,x,v)\dd v \leq E_0,\label{e:E0}\\
\int_{\R^d} f(t,x,v) \log f(t,x,v) \dd v \leq H_0.\label{e:H0}
\end{align}
In the space homogeneous case, because of the conservation of mass and energy, and the monotonicity of the entropy, it is not necessary to make the assumptions \eqref{e:M0}, \eqref{e:E0} and \eqref{e:H0}. It would suffice to require the initial data to have finite mass, energy and entropy. It is currently unclear whether these hydrodynamic quanitites will stay under control for large times and away from equilibrium in the space inhomogeneous case. Thus, at this point, it is simply an assumption we make.

We now state our main results. Our first theorem makes no further assumption on the initial data $f_{in}:\R^{2d}\to\R$ beyond what is required for a weak solution to exist in $[0,T_0]$.

\begin{theorem}\label{t:decay-generation}
Let $\gamma \in (-2,0]$. If $f:[0,T_0]\times \R^{2d}\to \R$ is a bounded weak solution of \eqref{e:divergence} satisfying \eqref{e:M0}, \eqref{e:E0}, and \eqref{e:H0}, then there exists $K_0>0$ such that $f$ satisfies
\begin{equation}\label{e:decay}
f(t,x,v) \leq K_0 \left(1+t^{-d/2}\right) (1+|v|)^{-1},
\end{equation}
for all $(t,x,v)\in [0,T_0]\times \R^{2d}$. The constant $K_0$ depends on $d$, $\gamma$, $m_0$, $M_0$, $E_0$, and $H_0$.
\end{theorem}

Note that even though we work with a bounded weak solution $f$, none of the constants in our estimates depend on $\|f\|_{L^\infty}$. Note also that our estimate does not depend on $T_0$. We use a definition of weak solution for which the estimates in \cite{golse2016} apply, since that is the main tool in our proofs.

We will show in Theorem \ref{t:polynomial} that an estimate of the form \eqref{e:decay} cannot hold with a power of $(1+|v|)$ less than $-(d+2)$, which also implies there is no \emph{a priori} exponential decay. On the other hand, if $f_{in}$ satisfies a Gaussian upper bound in the velocity variable, this bound is propagated:
\begin{theorem}\label{t:propagation-gaussian}
Let $f: [0,T_0]\times \R^{2d}\to \R$ be a bounded weak solution of the Landau equation \eqref{e:divergence} such that $f_{in}(x,v) \leq C_0 e^{-\alpha|v|^2}$, for some $C_0>0$ and a sufficiently small $\alpha>0$. Then
\[f(t,x,v) \leq C_1e^{-\alpha|v|^2},\]
where $C_1$ depends on $C_0$, $\alpha$, $d$, $\gamma$, $m_0$, $M_0$, $E_0$ and $H_0$. The value of $\alpha$ must be smaller than some $\alpha_0>0$ that depends on $\gamma$, $d$, $m_0$, $M_0$, $E_0$ and $H_0$.
\end{theorem}
This estimate is also independent of $T_0$. As a consequence of Theorem \ref{t:propagation-gaussian}, we will show in Theorem \ref{t:holder} that in this regime, $f$ is uniformly H\"older continuous on $[t_0,T_0]\times \R^{2d}$ for any $t_0\in (0,T_0)$. 

Note that under some formal asymptotic regime, the hydrodynamic quantities of the inhomogeneous Landau equations converge to solutions of the compressible Euler equation \cite{bardos1991}, which is known to develop singularities in finite time. Should we expect singularities to develop in finite time for the inhomogeneous Landau equation as well? That question seems to be out of reach with current techniques. A more realistic project is to prove that the solutions stay smooth for as long as the hydrodynamic quantities stay under control (as in \eqref{e:M0}, \eqref{e:E0} and \eqref{e:H0}). The results in this paper are an important step forward in that program.

\subsection{Related work} It was established in \cite{liu2014regularization} that solutions to \eqref{e:divergence} become $C^\infty$ smooth in all three variables conditionally to the solution being away from vacuum, bounded in $H^8$ (in the $d=3$ case) and having infinitely many finite moments. It would be convenient to extend this conditional regularity result to have less stringent assumptions. In particular, the assumptions \eqref{e:M0}, \eqref{e:E0} and \eqref{e:H0} are a much weaker assumption, which is also in terms of physically relevant hydrodynamic quantities. In \cite{golse2016}, the authors show how their local H\"older continuity result for linear kinetic equations with rough coefficients can be applied to solutions of the Landau equation provided that \eqref{e:M0}, \eqref{e:E0} and \eqref{e:H0} hold and in addition the solution $f$ is assumed to be bounded. While we also assume boundedness of $f$, our results do not quantitatively rely on this and in addition tell us some information about the decay for large velocities.


The local estimates for parabolic kinetic equations with rough coefficients play an important role in this work. Local $L^\infty$ estimates were obtained in \cite{pascucci2004ultraparabolic} using Moser iteration, and local H\"older estimates were proven in \cite{wang2009ultraparabolic, wang2011ultraparabolic} using a weak Poincar\'e inequality. A new proof was given in \cite{golse2016} using a version of De Giorgi's method.

Classical solutions for \eqref{e:divergence} have so far only been constructed in a close-to-equilibrium setting: see the work of Guo \cite{guo2002periodic} and Mouhot-Neumann \cite{mouhot2006equilibrium}. A suitable notion of weak solution, for general initial data, was constructed by Alexandre-Villani \cite{alexandre2004landau, villani1996global}.

The global $L^\infty$ estimate we prove in Theorem \ref{t:decay-generation} is similar to an estimate in \cite{luis-Boltzmann} for the Boltzmann equation. The techniques in the proof are completely different. The propagation of Gaussian bounds that we give in Theorem \ref{t:propagation-gaussian} is reminiscent of the result in \cite{gamba2009}. That result is for the space-homogeneous Boltzmann equation with cut-off, which is in some sense the opposite of the Landau equation in terms of the angular singularity in the cross section.

In order to keep track of the constants for parabolic regularization estimates (as in \cite{golse2016}) for large velocities, we describe a change of variables in Lemma \ref{l:T}. This change of variables may be useful in other contexts. It is related to one mentioned in the appendix of \cite{imbert2016weak} for the Boltzmann equation.

For the homogeneous Landau equation, which arises when $f$ is assumed to be independent of $x$ in \eqref{e:divergence}, the theory is more developed. The $C^\infty$ smoothing is established for hard potentials in \cite{desvillettes2000landau} and for Maxwell molecules in \cite{villani1998landau}, under the assumption that the initial data has finite mass and energy. Propagation of $L^p$ estimates in the case of moderately soft potentials was shown in \cite{wu2014global} and \cite{alexandre2015apriori}. Global upper bounds in a weighted $L^1_t(L^3_v)$ space were established in \cite{desvillettes2015landau}, even for $\gamma=-3$, as a consequence of entropy dissipation. Global $L^\infty$ bounds that do not depend on $f_{in}$ and that do not degenerate as $t\to\infty$ were derived in \cite{silvestre2015landau} for moderately soft potentials, and this result also implies $C^2$ smoothing by standard parabolic regularity theory. 

Note that in the space homogeneous case our assumptions \eqref{e:M0}, \eqref{e:E0} and \eqref{e:H0} hold for all $t>0$ provided that the initial data has finite mass, energy and entropy. Both Theorems \ref{t:decay-generation} and \ref{t:propagation-gaussian} are new results even in the space homogeneous case. The previous results for soft potentials do not address the decay of the solution for large velocities.


\subsection{Organization of the paper} In Section \ref{s:prelim}, we establish precise bounds on the coefficients $\bar a$, $\bar b$, and $\bar c$ in \eqref{e:divergence}. In Section \ref{s:local}, we derive the local estimates we will use to prove Theorem \ref{t:decay-generation}, starting from the Harnack estimate of \cite{golse2016}. Section \ref{s:global} contains the proof of Theorem \ref{t:decay-generation} and a propagating lower bound that implies the exponent of $(1+|v|)$ in \eqref{e:decay} cannot be arbitrarily high. In Section \ref{s:gaussian}, we prove Theorem \ref{t:propagation-gaussian} and the H\"older estimate, Theorem \ref{t:holder}. In Appendix \ref{s:A}, we derive a convenient maximum principle for kinetic Fokker-Planck equations.

\subsection{Notation} We say a constant is universal if it depends only on $d$, $\gamma$, $m_0$, $M_0$, $E_0$, and $H_0$. The notation $A\lesssim B$ means that $A\leq CB$ for a universal constant $C$, and $A \approx B$ means that $A\lesssim B$ and $B\lesssim A$. We will let $z=(t,x,v)$ denote a point in $\R_+\times \R^d\times \R^d$. For any $z_0=(t_0,x_0,v_0)$, define the Galilean transformation
\[\mathcal S_{z_0}(t,x,v) := (t_0+t, x_0 + x +tv_0,v_0+v).\]
We also have
\[\mathcal S_{z_0}^{-1}(t,x,v) := (t-t_0, x - x_0 -(t-t_0)v_0,v-v_0).\]
For any $r>0$ and $z_0 = (t_0,x_0,v_0)$, let
\[ Q_r(z_0) := (t_0-r^2,t_0] \times \{x : |x-x_0 - (t-t_0) v_0| < r^3 \}\times B_r(v_0),  \]
and $Q_r = Q_r(0,0,0)$. The shift $\mathcal S_{z_0}$ and the scaling of $Q_r$ correspond to the symmetries of the left-hand side of \eqref{e:divergence}. We will sometimes write $\partial_i$ or $\partial_{ij}$, and these will always refer to differentiation in $v$.

\section{The coefficients of the Landau equation}\label{s:prelim}

In this section we review various estimates of the coefficients $\bar a$, $\bar b$ and $\bar c$ in \eqref{e:divergence}. In calculating these upper and lower bounds, the dependence of $f$ on $t$ and $x$ is irrelevent, so in this section we will write $f(v)$ and $\bar a(v)$, etc.

\begin{lemma}\label{l:a}
Let $\gamma \in [-2, 0)$, and assume $f$ satisfies \eqref{e:M0}, \eqref{e:E0}, and \eqref{e:H0}. Then there exist constants $c$ and $C$ depending on $d$, $\gamma$, $m_0$, $M_0$, $E_0$, and $H_0$, such that for unit vectors $e\in \R^d$,
\begin{equation}\label{e:aij-lower}
 \bar a_{ij}(v) e_i e_j \geq c\begin{cases} (1+|v|)^{\gamma}, & e\in \mathbb S^{d-1}, \\
(1+|v|)^{\gamma+2},& e\cdot v = 0,\end{cases}
\end{equation}
and
\begin{equation}\label{e:aij-upper}
 \bar a_{ij}(v) e_i e_j \leq C\begin{cases} (1+|v|)^{\gamma+2}, &e\in \mathbb S^{d-1},\\
(1+|v|)^{\gamma},  & e\cdot v = |v|,\end{cases}
\end{equation}
where $\bar a_{ij}(v)$ is defined by \eqref{e:a}.
\end{lemma}

\begin{proof}
The lower bounds \eqref{e:aij-lower} are proven in \cite[Lemma 3.1]{silvestre2015landau}. For the upper bounds, the formula \eqref{e:a} implies
\begin{align*}
\bar a_{ij}(v)e_i e_j &= a_{d,\gamma} \int_{\R^d} \left(1 - \left(\frac {w\cdot e}{|w|}\right)^2\right)|w|^{\gamma+2} f(v-w)\dd w \\
&\lesssim \int_{\R^d} |w|^{\gamma+2} f(v-w)\dd w\\
&= \int_{\R^d} |v-z|^{\gamma+2} f(z)\dd z\\
&\lesssim \int_{\R^d} (|v|^{\gamma+2} + |z|^{\gamma+2})f(z)\dd z\\
&\lesssim  M_0(1+|v|^{\gamma+2}) +E_0,
\end{align*}
since $0\leq\gamma+2\leq 2$.

The above bound is valid for all $e\in \mathbb S^{d-1}$. If $e$ is parallel to $v$, then
\begin{align*}
 \int_{\R^d} \left(1 - \left(\frac {w\cdot e}{|w|}\right)^2\right)|w|^{\gamma+2} f(v-w)\dd w & = \int_{\R^d} \left(1 - \left(\frac {(v-z)\cdot e}{|v-z|}\right)^2\right)|v-z|^{\gamma+2} f(z)\dd z\\
&=  \int_{\R^d} \left(|v-z|^2 - \left(|v|-z\cdot e\right)^2\right)|v-z|^{\gamma} f(z)\dd z\\
&=  \int_{\R^d} \left(|z|^2 - (z\cdot e)^2\right)|v-z|^{\gamma} f(z)\dd z\\
&= \int_{\R^d} |z|^2\sin^2\theta |v-z|^{\gamma} f(z)\dd z,\end{align*}
where $\theta$ is the angle between $v$ and $z$. Let $R = |v|/2$. If $z\in B_R(v)$, then $|\sin \theta| \leq |v-z|/|v|$, and
\begin{align*}
\int_{B_R(v)} |z|^2\sin^2\theta |v-z|^\gamma f(z)\dd z &\leq \int_{B_R(v)} |z|^2 |v|^{-2}  |v-z|^{\gamma+2} f(z)\dd z\\
&\leq \frac {|v|^{\gamma}}{2^{\gamma+2}}\int_{B_R(v)}|z|^2f(z)\dd z \lesssim E_0|v|^{\gamma}.
\end{align*}
If $|v-z| \geq R=|v|/2$, then $|v-z|^\gamma \lesssim |v|^\gamma$, and we have
\begin{align*}
\int_{\R^d\setminus B_R(v)} |z|^2\sin^2\theta |v-z|^\gamma f(z)\dd z &\lesssim  |v|^{\gamma} \int_{\R^d\setminus B_R(v)} |z|^2 f(z)\dd z \lesssim E_0 |v|^{\gamma}.
\end{align*}
\end{proof}

In the proof of Theorem \ref{t:decay-generation}, we will need to keep track of how the bounds on $\bar b$ and $\bar c$ in the next two lemmas depend on the local $L^\infty$ norm of $f$. In Lemma \ref{l:c} and Lemma \ref{l:b}, $\|f\|_{L^\infty(A)}$ means $\|f(t,x,\cdot)\|_{L^\infty(A)}$ for any set $A\subseteq\R^d$.

\begin{lemma}\label{l:c}
Let $f$ satisfy \eqref{e:M0}, \eqref{e:E0}, and \eqref{e:H0}. Then $\bar c(v)$ defined by \eqref{e:c} satisfies
\begin{equation*}
\bar c(v) \lesssim \begin{cases} (1+|v|)^\gamma(1+\|f\|_{L^\infty(B_{\rho}(v))})^{-\gamma/d}, &\dfrac{-2d}{d+2}\leq \gamma < 0,\\
(1+|v|)^{-2-2\gamma/d}\left(1+\|f\|_{L^\infty(B_{\rho}(v))}  \right)^{-\gamma/d}, &-d < \gamma < \dfrac{-2d}{d+2},\end{cases}
\end{equation*}
where the constants 
depend on $d, \gamma, M_0,$ and $E_0$, and
\[\rho = \begin{cases} 1, &|v|<2,\\
|v|^{-2/d}, &|v|\geq 2.
\end{cases}
\]

\end{lemma}

\begin{proof}
%
Assume first $|v|\geq 2$. Let $r := |v|^{-2/d} (1+\|f\|_{L^\infty(B_\rho(v))})^{-1/d} < \rho$. Consider
\begin{align*}
I_1 = \int_{B_{r}} |w|^\gamma  &  f(v-w)\dd w, \quad I_2 = \int_{B_{|v|/2}\setminus B_{r} } |w|^\gamma f(v-w)\dd w,
\\& \quad I_3 = \int_{\R^d \setminus B_{|v|/2}} |w|^\gamma f(v-w)\dd w.
\end{align*}
We have
\[I_1 \lesssim \|f\|_{L^\infty(B_\rho(v)) } r^{d+\gamma} \lesssim |v|^{-2-2\gamma/d} \|f\|_{L^\infty(B_\rho(v))}^{-\gamma/d} .\]
\begin{align*}
I_2 \lesssim r^\gamma |v|^{-2}\int\limits_{B_{|v|/2}} |v-w|^2 f(v-w) \dd w &\lesssim E_0 |v|^{-2-2\gamma/d} (1+\|f\|_{L^\infty(B_\rho(v))})^{-\gamma/d}.
 \end{align*}

Finally, for $|w|\geq |v|/2$, we have $|w|^\gamma \lesssim |v|^{\gamma}$, and
\begin{align*}
I_3 \lesssim |v|^\gamma \int_{\R^d\setminus B_{|v|/2}}f(v-w)\dd w \leq M_0|v|^\gamma.
 \end{align*}
Thus $\bar c(v)\lesssim (1+\|f\|_{L^\infty(B_\rho(v))})^{-\gamma/d}|v|^{-2-2\gamma/d} + |v|^\gamma$ for $|v|>2$.

When $\gamma \in \left(-d, \dfrac{-2d}{d+2}\right)$, $-2-2\gamma/d > \gamma$ and we get
\[ \bar c(v)\lesssim (1+\|f\|_{L^\infty(B_\rho(v))})^{-\gamma/d}|v|^{-2-2\gamma/d}.\]
When $\gamma \in \left[\dfrac{-2d}{d+2} , 0\right)$, $\gamma > -2-2\gamma/d$ and we get
\[ \bar c(v)\lesssim (1+\|f\|_{L^\infty(B_\rho(v))})^{-\gamma/d}|v|^{\gamma}.\]
This completes the proof in the case $|v| > 2$.

For $|v| \leq 2$, $\gamma\in (-d, 0]$, and any $R\in (0,1]$ we have that
\begin{align*}
\int\limits_{\R^d} |w|^\gamma f(v-w) \dd w &= \int_{B_R} |w|^{\gamma} f(v-w)\dd w + \int_{\R^d\setminus B_R} |w|^\gamma f(v-w)\dd w, \\
&\lesssim R^{d+\gamma}\|f\|_{L^\infty(B_1(v))} + R^\gamma M_0.
\end{align*}
Choosing $R = (1+\|f\|_{L^\infty(B_1(v)) } )^{-1/d}$, we then have
\[\bar c(v) \lesssim (R^{d+\gamma}\|f\|_{L^\infty(B_1(v))} + R^\gamma M_0) \lesssim (1+ \|f\|_{L^\infty(B_1(v))})^{-\gamma/d},\] for $|v|\leq 2$, completing the proof.
\end{proof}


\begin{lemma}\label{l:b}
Let $f$ satisfy \eqref{e:M0}, \eqref{e:E0}, and \eqref{e:H0}.
 Then $\bar b(v)$ defined by \eqref{e:b} satisfies
the estimate
\begin{equation}\label{e:b-upper}
|\bar b(v)| \lesssim \begin{cases}(1+|v|)^{\gamma+1}(1+\|f\|_{L^\infty(B_{\rho}(v))})^{-(\gamma+1)/d}, &\gamma\in [-2,-1),\\[2ex]
(1+ |v|)^{\gamma+1}, &\gamma\in [-1,0]\end{cases}
\end{equation}
where the constants depend on $d, \gamma, M_0,$ and $E_0$, and
\[\rho = \begin{cases} 1, &|v|<2,\\
|v|^{-2/d}, &|v|\geq 2.\end{cases}
\]

\end{lemma}
\begin{proof}
Taking norms, we have
\[|\bar b(v)| \lesssim \int_{\R^d} |w|^{1+\gamma}f(v-w)\dd w.\]
If $\gamma \in [-2,-1)$, then $0>1+\gamma \geq -1 \geq \displaystyle\frac{-2d}{d+2}$, and the conclusion follows from Lemma \ref{l:c}. If $\gamma\in [-1,0]$, we have
\begin{align*}
|\bar b(v)| &\lesssim \int_{\R^d} (|v|^{\gamma+1} + |v-w|^{\gamma+1} ) f(v-w)\dd w \\
&\lesssim |v|^{\gamma+1} M_0 + E_0^{(1+\gamma)/2} M_0^{(1-\gamma)/2} \lesssim (1+ |v|)^{\gamma+1}.
\end{align*}
\end{proof}

\section{Local estimates}\label{s:local}

In this section we refine the local estimates in \cite{pascucci2004ultraparabolic} and \cite{golse2016} for linear kinetic equations with rough coefficients. Essentially, we start from their results and apply scaling techniques to improve the local $L^\infty$ estimates.

We will need the following technical lemma. See \cite[Lemma 4.3]{hanlin} for the proof.
\begin{lemma}\label{l:eta}
Let $\eta(r)\geq 0$ be bounded in $[r_0,r_1]$ with $r_0\geq 0$. Suppose for $r_0\leq r<R\leq r_1$, we have
\[\eta(r) \leq \theta \eta(R) + \frac A {(R-r)^\alpha} + B\]
for some $\theta \in [0,1)$ and $A,B,\alpha \geq 0$. Then there exists $c(\alpha,\theta)>0$ such that for any $r_0\leq r<R\leq r_1$, there holds
\[\eta(r) \leq c(\alpha,\theta)\left(\frac A {(R-r)^\alpha} + B\right).\]
\end{lemma}

\begin{proposition}\label{p:harnack}
If $g(t,x,v)\geq 0$ is a weak solution of
\begin{equation}\label{e:FP}
\partial_t g + v\cdot \nabla_x g = \nabla_v \cdot(A\nabla_v g) + B \cdot\nabla_v g + s
\end{equation}
in $Q_1$,
with
\begin{align*}
0 < \lambda I \leq A(t,x,v) \leq \Lambda I, \qquad &(t,x,v)\in Q_1,\\
|B(t,x,v)| \leq \Lambda,\qquad &(t,x,v)\in Q_1,\\
s\in L^\infty(Q_1), \qquad &
\end{align*}
then
\begin{equation}\label{e:harnack}
\sup_{Q_{1/2}} g \leq C\left(\|g\|_{L^\infty_{t,x}L^1_v(Q_1)} + \|s\|_{L^\infty(Q_1)}\right),
\end{equation}
with $C$ depending only on $d, \lambda$, and $\Lambda$.
\end{proposition}
\begin{proof}
It is proven in \cite{golse2016} that if $g(t,x,v)$ solves \eqref{e:FP} weakly with $A$, $B$, and $s$ as in the statement of the proposition, then
\begin{equation*}
\|g\|_{L^\infty(Q_{1/2})} \leq C \left(\|g\|_{L^2(Q_1)} + \|s\|_{L^\infty(Q_1)}\right),
\end{equation*}
with $C$ depending on $d, \lambda,$ and $\Lambda$. Since $\|g\|_{L^2(Q_1)} \leq \sqrt{\omega_d} \|g\|_{L_{t,x}^\infty L_v^2(Q_1)}$, where $\omega_d = \mathcal L^d(B_1)$, we also have
\begin{equation}\label{e:L2est}
\|g\|_{L^\infty(Q_{1/2})} \leq C \left(\|g\|_{L_{t,x}^\infty L_v^2(Q_1)} + \|s\|_{L^\infty(Q_1)}\right).
\end{equation}
To replace $\|g\|_{L_{t,x}^\infty L_v^2(Q_1)}$ with $\|g\|_{L_{t,x}^\infty L_v^1(Q_1)}$, we use an interpolation argument. For $0<r\leq 1$, define
\begin{equation}\label{e:rescale}
\begin{array}{ll}
g_r(t,x,v)  :=  g(r^2t, r^3x,rv), & s_r(t,x,v) := s(r^2t,r^3x,rv),\\
A_r(t,x,v)  :=  A(r^2t,r^3x, rv), & B_r(t,x,v) := B(r^2t,r^3x,rv),\end{array}
\end{equation}
and note that $g_r$ satisfies
\begin{equation}\label{e:gr}
\partial_t g_r + v\cdot \nabla_x g_r = \nabla_v \cdot(A_r\nabla_v g_r) + rB_r \cdot\nabla_v g_r + r^2s_r
\end{equation}
in $Q_1$. Since $r\leq 1$, we may apply \eqref{e:L2est} to $g_r$, which gives
\begin{equation}\label{e:Qr}
\|g\|_{L^\infty(Q_{r/2})} \leq C\left(\frac 1 {r^{d/2}}\|g\|_{L_{t,x}^\infty L_v^2(Q_r)} + r^2\|s\|_{L^\infty(Q_r)}\right),
\end{equation}
for any $r\in (0,1]$. Now, for $\theta,R\in (0,1)$, apply \eqref{e:Qr} in $Q_{(1-\theta)R}(z)$ for each $z\in Q_{\theta R}$ to obtain
\begin{align*}
\|g\|_{L^\infty(Q_{\theta R})} &\leq C\left(\frac 1 {[(1-\theta)R]^{d/2}} \|g\|_{L_{t,x}^\infty L_v^2(Q_R)} + R^2\|s\|_{L^\infty(Q_R)}\right)\\
&\leq C\left(\frac 1 {[(1-\theta)R]^{d/2}} \|g\|_{L_{t,x}^\infty L_v^2(Q_R)} + \|s\|_{L^\infty(Q_1)}\right).
\end{align*}
By the H\"older and Young inequalities, we have
\begin{align*}
\|g\|_{L^\infty(Q_{\theta R})} &\leq C\left(\frac 1 {[(1-\theta)R]^{d/2}} \|g\|_{L^\infty(Q_R)}^{1/2}\|g\|_{L_{t,x}^\infty L_v^1(Q_R)}^{1/2} + \|s\|_{L^\infty(Q_1)}\right)\\
&\leq \frac 1 2 \|g\|_{L^\infty(Q_R)} + C\left(\frac 1 {[(1-\theta)R]^{d}}\|g\|_{L_{t,x}^\infty L_v^1(Q_R)} + \|s\|_{L^\infty(Q_1)}\right).
\end{align*}
Define $\eta(\rho) = \|g\|_{L^\infty(Q_\rho)}$ for $\rho\in (0,1]$. Then for any $0<r<R\leq 1$, we have
\[\eta(r) \leq \frac 1 2 \eta(R) + \frac C {(R-r)^d} \|g\|_{L_{t,x}^\infty L_v^1(Q_1)} + C\|s\|_{L^\infty(Q_1)}.\]
Applying Lemma \ref{l:eta}, we obtain
\[\eta(r) \leq \frac C {(R-r)^d} \|g\|_{L_{t,x}^\infty L_v^1(Q_1)} + C\|s\|_{L^\infty(Q_1)}.\]
Let $R\to 1-$ and set $r=\frac 1 2$ to conclude \eqref{e:harnack}.
\end{proof}

\begin{lemma}\label{l:improved}
Let $g(t,x,v)$ solve \eqref{e:FP} weakly in $Q_R(z_0)$ for some $z_0\in \R^{2d+1}$ and $R>0$, with
\begin{align*}
0 < \lambda I \leq A(t,x,v) \leq \Lambda I, \qquad &(t,x,v)\in Q_R,\\
|B(t,x,v)| \leq \Lambda/R,\qquad &(t,x,v)\in Q_R,\\
s\in L^\infty(Q_R). \qquad &
\end{align*}
Then the improved estimate
\begin{equation}\label{e:improved}
g(t_0,x_0,v_0) \leq C\left(\|g\|_{L_{t,x}^\infty L_v^1(Q_R)}^{2/(d+2)} \|s\|_{L^\infty(Q_R)}^{d/(d+2)} + R^{-d}\|g\|_{L_{t,x}^\infty L_v^1(Q_R)}\right)
\end{equation}
holds, with $C$ depending only on $d, \lambda$, and $\Lambda$.
\end{lemma}
\begin{proof}
By applying the change of variables
\begin{equation*}
(t,x,v) \mapsto \left( \frac {t-t_0}{R^2}, \frac{x-x_0 - (t-t_0)v_0}{R^3}, \frac{v-v_0} R\right)
\end{equation*}
to $g$ and $s$, we may suppose $(t_0,x_0,v_0) = (0,0,0)$ and $R=1$.

For $r\in(0,1]$ to be determined, we make the transformation \eqref{e:rescale} as in the proof of Proposition \ref{p:harnack} and get a function $g_r$ satisfying \eqref{e:gr} in $Q_1$. Then Proposition \ref{p:harnack} implies
\begin{align*}
g(0,0,0) &\leq C\left(\|g_r\|_{L^\infty_{t,x}L^1_v(Q_1)} + \|r^2s_r\|_{L^\infty(Q_1)}\right)\\
&= C\left(r^{-d}\|g\|_{L^\infty_{t,x}L^1_v(Q_r)} + r^2\|s\|_{L^\infty(Q_r)}\right)\\
&\leq C\left(r^{-d}\|g\|_{L^\infty_{t,x}L^1_v(Q_1)} + r^2\|s\|_{L^\infty(Q_1)}\right).
\end{align*}
If $\|g\|_{L_{t,x}^\infty L_v^1(Q_1)} \leq \|s\|_{L^\infty(Q_1)}$, then the choice $r = (\|g\|_{L^\infty_{t,x}L_v^1(Q_1)}/\|s\|_{L^\infty(Q_1)})^{1/(d+2)}$ implies
\[g(0,0,0) \leq C\|g\|_{L_{t,x}^\infty L_v^1(Q_1)}^{2/(d+2)} \|s\|_{L^\infty(Q_1)}^{d/(d+2)}.\]
On the other hand, if $\|s\|_{L^\infty(Q_1)}\leq \|g\|_{L_{t,x}^\infty L_v^1(Q_1)}$, the choice $r=1$ implies $g(0,0,0) \leq C\|g\|_{L_{t,x}^\infty L_v^1(Q_1)}$, so we have
\begin{equation*}
g(0,0,0) \leq C\left(\|g\|_{L_{t,x}^\infty L_v^1(Q_1)}^{2/(d+2)} \|s\|_{L^\infty(Q_1)}^{d/(d+2)} + \|g\|_{L_{t,x}^\infty L_v^1(Q_1)}\right)
\end{equation*}
in both cases.
\end{proof}

\section{Global estimates}\label{s:global}

In this section, we prove global upper bounds for solutions $f$ of \eqref{e:divergence}. Our bounds depend only on the estimates on the hydrodynamic quantities \eqref{e:M0}, \eqref{e:E0} and \eqref{e:H0}. Our bound does not depend on an upper bound of the initial data. We also get that the solution will have certain polynomial decay in $v$ for $t>0$.

From Lemma \ref{l:a}, we see that the bounds on $\bar a_{ij}(t,x,v)$ degenerate as $|v|\to\infty$. In the first lemma, we show how to change variables to obtain an equation with uniform ellipticity constants independent of $|v|$.
\begin{lemma}\label{l:T}
Let $z_0 =(t_0,x_0,v_0)\in \R_+\times \R^{2d}$ be such that $|v_0|\geq 2$, and let $T$ be the linear transformation such that
\begin{equation*}
T e = \begin{cases}   |v_0|^{1+\gamma/2} e , & e \cdot v_0 = 0\\
|v_0|^{\gamma/2}e, & e \cdot v_0 = |v_0|.\end{cases}
\end{equation*}
Let $\tilde T(t,x,v) = (t,Tx,Tv)$, and define
\begin{align*}
\mathcal T_{z_0}(t,x,v) &:= \mathcal S_{z_0} \circ \tilde T (t,x,v)\\
& = (t_0+t,x_0+T x + t v_0 ,v_0 + T v).
\end{align*}

Then,
\begin{enumerate}
\item[(a)] There exists a constant $C>0$ independent of $v_0\in\R^d\setminus B_2$ such that for all $v\in B_1$,
\[ C^{-1} |v_0| \leq |v_0 + Tv| \leq C |v_0|.\]
\item[(b)] If $f_T(t,x,v) := f(\mathcal T_{z_0}(t,x,v))$, then $f_T$ satisfies
\begin{equation}\label{e:isotropic}
\partial_t f_T + v \cdot \nabla_x f_T = \nabla_v \left[ A(z)\nabla_v f_T\right] +  B(z)\cdot \nabla_v f_T +  C(z) f_T
\end{equation}
in $Q_R$ for any $0<R<\min\{\sqrt{t_0},c_1 |v_0|^{-1-\gamma/2}\}$, where $c_1$ is a universal constant, and
\begin{align*}
\lambda I &\leq A(z) \leq \Lambda I,\\
|B(z)| &\lesssim \begin{cases} |v_0|^{1+\gamma/2}\left(1+\|f(t,x,\cdot)\|_{L^\infty(B_\rho(v))}\right)^{-(\gamma+1)/d}, &\gamma\in [-2,-1),\\[2ex]
|v_0|^{1+\gamma/2}, &\gamma\in [-1,0],\end{cases}\\
|C(v)| &\lesssim \begin{cases} |v_0|^\gamma\left(1+\|f(t,x,\cdot)\|_{L^\infty(B_\rho(v))}\right)^{-\gamma/d}, &\dfrac{-2d}{d+2}\leq \gamma < 0,\\
|v_0|^{-2-2\gamma/d}\left(1+\|f(t,x,\cdot)\|_{L^\infty(B_\rho(v))}\right)^{-\gamma/d}, &-2 < \gamma < \dfrac{-2d}{d+2},\end{cases}
\end{align*}
with $\lambda$ and $\Lambda$ universal, and $\rho \lesssim 1+ |v_0|^{-2/d}$. 
\end{enumerate}
\end{lemma}
\begin{proof}
Since $|v|\leq 1$ and $|v_0| > 2$,
\[|v_0| - |v_0|^{1+\gamma/2} \leq |v_0| - |Tv| \leq |v_0 + Tv| \leq |v_0| + |Tv| \leq |v_0| + |v_0|^{1+\gamma/2}.\]
 Thus, (a) follows since $\gamma \in (-2,0)$.

For (b), by direct computation, $f_T$ satisfies \eqref{e:isotropic} with
\[A(z) = T^{-1}\bar a(\mathcal T_{z_0}(z)) T^{-1}, \quad B(z) = T^{-1}\bar b(\mathcal T_{z_0}(z)), \quad C(z) = \bar c(\mathcal T_{z_0}(z)).\]
In order to keep the proof clean, let us write $\bar a_{ij}$ and $A_{ij}$ instead of $\bar a_{ij}(\mathcal T_{z_0}(z))$ and $A_{ij}(z)$ for the rest of the proof.

Fix $z = (t,x,v)\in Q_R$, and let $\tilde v = v_0+T v$. From part (a), we know that $|\tilde v| \approx |v_0|$. Applying Lemma \ref{l:a}, we have that for any unit vector $e$,
\begin{equation}\label{e:aijT}
\bar a_{ij} e_ie_j \lesssim \begin{cases}(1+|v_0|)^{\gamma}, &e = \tilde v / |\tilde v|,\\
(1+|v_0|)^{\gamma+2}, & e \in S^{d-1}. \end{cases}
\end{equation}
and,
\begin{equation}\label{e:aijT2}
\bar a_{ij} e_ie_j \gtrsim \begin{cases}(1+|v_0|)^{\gamma}, &e \in S^{d-1},\\
(1+|v_0|)^{\gamma+2}, & e \cdot \tilde v = 0. \end{cases}
\end{equation}

Our first step is to verify that we can switch $\tilde v$ for $v_0$ in \eqref{e:aijT} and \eqref{e:aijT2}.

Let us start with \eqref{e:aijT}. This is where the assumption $|v| < R \leq C_1 |v_0|^{-1-\gamma/2}$ plays a role. We can choose $c_1$ so as to ensure that $|Tv| \leq 1$. Since $v_0 = \tilde v - Tv$ and using the fact that $\bar a_{ij}$ is positive definite,
\[ \bar a_{ij} (v_0)_i (v_0)_j \leq 2 \bar a_{ij} \tilde v_i \tilde v_j + 2 \bar a_{ij}(Tv)_i (Tv)_j \leq C |v_0|^{2+\gamma}.\]
Let $e_0 = v_0 / |v_0|$. The computation above tells us that $\bar a_{ij} (e_0)_i (e_0)_j \lesssim |v_0|^\gamma$.

Let us now turn to \eqref{e:aijT2}. We will show that
\begin{equation} \label{e:aijT2v0}
 \bar a_{ij} w_iw_j \gtrsim (1+|v_0|)^{\gamma+2} |w|^2 \qquad \text{if } w \cdot v_0 = 0.
\end{equation}
Note that $(1+|v_0|)^{2+\gamma}$ and $(1+|v_0|)^\gamma$ are comparable when $|v_0|$ is small, so we only need to verify \eqref{e:aijT2v0} for $w \cdot v_0 = 0$ and $|v_0|$ arbitrarily large. For such vector $w$, we write $w= \eta \tilde v + w'$ with $w'\cdot \tilde v = 0$. Since $|\tilde v - v_0| = |Tv| \leq 1$, we have $|\eta| = |w\cdot \tilde v| / |\tilde v|^2 = |w\cdot(\tilde v - v_0)|/|\tilde v|^2 \leq |w| |\tilde v|^{-2}$. Moreover, $|w'| \approx |w|$.

Since $\bar a_{ij}$ is positive definite,
\[\bar a_{ij} (\sqrt 2 \eta \tilde v - w'/\sqrt 2)_i (\sqrt 2 \eta \tilde v - w'/ \sqrt 2)_j \geq 0,\]
then we have
\begin{align*}
\bar a_{ij} w_i w_j &\geq \frac 1 2 \bar a_{ij} w'_iw'_j - \eta^2\bar a_{ij} \tilde v_i \tilde v_j\\
&\geq \left( c(1+|v_0|)^{\gamma+2} - (1+|v_0|)^{\gamma} \right)|w|^2 \gtrsim (1+|v_0|)^{\gamma+2}  |w|^2,
\end{align*}
as desired.


Let $w \in \R^d$ be arbitrary. We will estimate $A_{ij} w_i w_j$ from above. Writing $w = \mu e_0 + \tilde w$, with $\tilde w \cdot e = 0$.
\begin{align*}
A_{ij} w_i w_j &= |v_0|^{-\gamma} \left( \mu^2 \bar a_{ij} (e_0)_i (e_0)_j + 2 \mu |v_0|^{-1} \bar a_{ij} (e_0)_i \tilde w_j + |v_0|^{-2} \bar a_{ij} \tilde w_i \tilde w_j \right),\\
\intertext{and using that $\bar a_{ij}$ is positive definite,}
A_{ij}w_iw_j &\leq 2|v_0|^{-\gamma} \left( \mu^2 \bar a_{ij} (e_0)_i (e_0)_j + |v_0|^{-2} \bar a_{ij} \tilde w_i \tilde w_j \right),\nonumber\\
&\leq C \left( \mu^2 + |\tilde w|^2 \right) =: \Lambda |w|^2.\nonumber
\end{align*}
This establishes upper bound $\{A_{ij}\} \leq \Lambda I$ for some $\Lambda>0$.

Now we will prove the lower bound for $A_{ij}$. Again, we write $w = \mu e_0 + \tilde w$ with $e_0 \cdot \tilde w = 0$. We need to analyze the quadratic form associated with the coefficients $\bar a_{ij}$ more closely. From \eqref{e:aijT2}, we have that for some universal constant $c>0$,
\[ c |v_0|^\gamma (\mu^2 + |\tilde w|^2) \leq \bar a_{ij} w_i w_j = \mu^2 \bar a_{ij} (e_0)_i (e_0)_j + 2 \mu \bar a_{ij} (e_0)_i \tilde w_j + \bar a_{ij} \tilde w_i \tilde w_j. \]
Moreover, \eqref{e:aijT} implies that there is a universal constant $\delta > 0$ so that
\[ c |v_0|^\gamma (\mu^2 + |\tilde w|^2) \geq \delta \mu^2 \bar a_{ij} (e_0)_i (e_0)_j + \delta |v_0|^{-2} \bar a_{ij} \tilde w_i \tilde w_j. \]
Subtracting the two inequalities above,
\[ (1-\delta) \mu^2 \bar a_{ij} (e_0)_i (e_0)_j + 2 \mu \bar a_{ij} (e_0)_i \tilde w_j + (1-\delta |v_0|^{-2}) \bar a_{ij} \tilde w_i \tilde w_j \geq 0.\]
The same inequality holds if we replace $w = \mu e_0 + \tilde w$ with $w = (1-\delta/2)^{-1/2} \mu e_0 + (1-\delta/2)^{1/2}|v_0|^{-1} \tilde w$, therefore
\[ \frac{1-\delta}{1-\delta/2} \mu^2 \bar a_{ij} (e_0)_i (e_0)_j + 2 \mu |v_0|^{-1} \bar a_{ij} (e_0)_i \tilde w_j + (1-\delta/2)(1-\delta |v_0|^{-2}) |v_0|^{-2} \bar a_{ij} \tilde w_i \tilde w_j \geq 0.\]
Recalling the formula above for $A_{ij} w_i w_j$, and replacing it in the left hand side, we get
\[ A_{ij} w_i w_j - \left( 1 - \frac{1-\delta}{1-\delta/2} \right) |v_0|^{-\gamma} \mu^2 \bar a_{ij} (e_0)_i (e_0)_j - \left( 1 - (1-\delta/2)(1-\delta |v_0|^{-2}) \right) |v_0|^{-2-\gamma} \bar a_{ij} \tilde w_i \tilde w_j \geq 0.\]
Therefore, using \eqref{e:aijT2} and \eqref{e:aijT2v0},
\begin{align*}
 A_{ij} w_i w_j &\geq \left( 1 - \frac{1-\delta}{1-\delta/2} \right) |v_0|^{-\gamma} \mu^2 \bar a_{ij} (e_0)_i (e_0)_j + \left( 1 - (1-\delta/2)(1-\delta |v_0|^{-2}) \right) |v_0|^{-2-\gamma} \bar a_{ij} \tilde w_i \tilde w_j, \\
&\geq \lambda (\mu^2 + |\tilde w|^2),
\end{align*}
for some universal constant $\lambda > 0$. This establishes the lower bound $\{A_{ij}\} \geq \lambda I$.

\commentout{
\medskip
\medskip
\medskip
\medskip

We now turn to the lower bound for

Next, \eqref{e:aijT} and \eqref{e:aijT2} imply that there is a unit eigenvector $\bar e$ of $\bar a_{ij}$ with eigenvalue $\bar \lambda \approx (1+|v_0|)^\gamma$. For $|v_0|$ sufficiently large, the angle between $e_0$ and $\bar e$ must be small. Indeed, writing $e_0 = (e_0\cdot \bar e)\bar e + e^{\perp}$, we have $\bar a_{ij}(e_0)_i(e_0)_j = (e_0\cdot\bar e)^2 \bar \lambda + \bar a_{ij} e^{\perp}_i e^{\perp}_j$, so that \eqref{e:aijT} and \eqref{e:aijT2} imply $|e^{\perp}|^2 \lesssim |v_0|^{-2}$, and therefore, $|e_0 - \bar e|\lesssim |v_0|^{-1}$. Similarly, for $\tilde w$ perpendicular to $e_0$, we can write $\tilde w = (\tilde w\cdot \bar e)\bar e + \bar w$ with $|\tilde w - \bar w| \lesssim |\tilde w| |v_0|^{-1}$. Now, with $w\in\R^d$ and $w = \mu e_0 + \tilde w$ as above, we have
\begin{align*}
\bar a_{ij}(e_0)_i \tilde w_j &= \bar a_{ij}\bar e_i (\tilde w - \bar w)_j + \bar a_{ij} (e_0 - \bar e)_i \tilde w_j\\
&\geq  - \bar \lambda |\tilde w - \bar w| - |\bar a (e_0-\bar e)| |\tilde w|\\
&\geq -c(1+|v_0|)^\gamma |v_0|^{-1} |\tilde w|,
\end{align*}
for some constant $c$. Then, if $|v_0|\geq \rho_0$ large enough, \eqref{e:Aijw} implies
\begin{align*}
A_{ij} w_i w_j &\geq |v_0|^{-\gamma}\left( \mu^2 |v_0|^{\gamma} - 2c\mu |v_0|^{-2+\gamma}|\tilde w| + |v_0|^{\gamma}|\tilde w|^2\right)\\
&\geq \mu^2 + |\tilde w|^2 = |w|^2.
\end{align*}

\medskip
\medskip
\medskip
\medskip
\medskip
}

To derive the bound on $B(z)$, Lemma \ref{l:b} and conclusion (a) imply
\begin{align*}
|B(z)| &\lesssim \|T^{-1}\| |\bar b(\mathcal T_{z_0}(z))|\\
&\lesssim \begin{cases} (1+|v_0|)^{\gamma/2+1}(1+\|f\|_{L^\infty(B_{\rho'}(\tilde v))})^{-(\gamma+1)/d}, &\gamma \in (-2,-1),\\
(1+|v_0|)^{\gamma/2+1}, &\gamma \in [-1,0],\end{cases}
\end{align*}
where $\rho' = |\tilde v|^{-2/d}$. From the triangle inequality, we have that $B_{\rho'}(\tilde v) \subset B_\rho(v_0)$, with $\rho \lesssim (1+|v_0|)^{-2/d}+R(1+|v_0|)^{(\gamma+2)/2} \leq 1+(1+|v_0|)^{-2/d}$.
The bound on $C(z)$ follows in a similar manner, using Lemma \ref{l:c}.
\end{proof}

The key lemma in the proof of Theorem \ref{t:decay-generation} is the following pointwise estimate on $f$:
\begin{lemma}\label{l:gain}
Let $\gamma\in (-2,0]$, $T_0>0$, and let $f:[0,T_0]\times \R^{2d}\to\R_+$ solve the Landau equation \eqref{e:divergence} weakly. If
\[f(t,x,v) \leq K (1+t^{-d/2})(1+|v|)^{-\alpha}\]
in $[0,T_0]\times \R^{2d}$ for some $\alpha\in [0,1]$ and $K\geq 1$, then
\begin{equation}\label{e:pointwise}
f(t,x,v) \leq C\left((K (1+t^{-d/2}))^{(d-\gamma)/(d+2)} (1+|v|)^{P(d,\alpha,\gamma)} + K^{Q(\gamma)}(1+ t^{-d/2}) (1+|v|)^{-1}\right),
\end{equation}
for some $C$ universal and
\[P(d,\alpha,\gamma) = \begin{cases} -1 - d(1+\alpha)/(d+2), &\gamma\in \left[\dfrac{-2d}{d+2}, 0\right],\\
-[d(4+\gamma)+2 +2\gamma + \alpha d]/(d+2), &\gamma\in \left(-2, \dfrac{-2d}{d+2} \right),\end{cases}\]
\[Q(\gamma) = \begin{cases} 0, &\gamma \in [-1,0]\\
-(1+\gamma), &\gamma\in (-2,-1).\end{cases}\]
\end{lemma}
\begin{proof}
\emph{Case 1: $\gamma \in [-1,0]$.} 
Let $z_0 = (t_0,x_0,v_0)$ be such that such that $|v_0|\geq 2$. Define $r_0 = \min\{1,\sqrt{t_0}\}$, and note that $r_0^{-d} \approx (1+t_0^{-d/2})$. Letting $f_T$ be as in Lemma \ref{l:T}, we will estimate $f_T(t,x,v)$ in $Q_R$, where
\[R := c_1(r_0/2)(1+|v_0|)^{-(2+\gamma)/2},\]
with $c_1$ as in Lemma \ref{l:T}(b). We have that $f_T$ solves \eqref{e:isotropic} in $Q_{R}$, and by Lemma \ref{l:T}(a) and our assumption on $f$,
\begin{equation}\label{e:QRbound}
f_T(t,x,v) \lesssim K r_0^{-d} (1+|v_0|)^{-\alpha}
\end{equation}
in $Q_R$. Feeding \eqref{e:QRbound} into Lemma \ref{l:T}(b), we have
\begin{align}
0<&\lambda I \leq  A(z) \leq \Lambda I,\nonumber\\
|B(z)|&\lesssim (1+|v_0|)^{(2+\gamma)/2},\label{e:Bbound}\\
|C(z)| &\lesssim   \left(Kr_0^{-d}\right)^{-\gamma/d}(1+|v_0|)^{\gamma},\label{e:Cbound}
\end{align}
in $Q_R$.




Let $Q_{T,R}$ be the image of $Q_R$ under $z \mapsto \mathcal T_{z_0}(z)$, and note that
\begin{align}
\| f_T \|_{L_{t,x}^\infty L_v^1(Q_R)} &= \det(T^{-1})\|f\|_{L_{t,x}^\infty L_v^1(Q_{T,R})}\nonumber\\
& = (1+|v_0|)^{-[(d-1)(2+\gamma)/2 + \gamma/2]}\|f\|_{L_{t,x}^\infty L_v^1(Q_{T,R})}\nonumber\\
&\leq (1+|v_0|)^{-\left(1+d(2+\gamma)/2\right)}E_0,\label{e:L1bound}
\end{align}
where the last inequality comes from the energy bound \eqref{e:E0} and Lemma \ref{l:T}(a).

By \eqref{e:Bbound} and our choice of $R$, we can apply Lemma \ref{l:improved} in $Q_R$ with $g = f_T$ and $s = C(z) f_T$ to obtain
\begin{align}\label{e:estimate-small-gamma}
f(t_0,x_0,v_0) &\leq C\left(\|f_T\|_{L_{t,x}^\infty L_v^1(Q_R)}^{2/(d+2)} \| C(z) f_T\|_{L^\infty(Q_R)}^{d/(d+2)}+ r_0^{-d}(1+|v_0|)^{d(2+\gamma)/2}\|f_T\|_{L_{t,x}^\infty L_v^1(Q_R)}\right)\nonumber\\
&\leq C\left((Kr_0^{-d})^{(d-\gamma)/(d+2)} (1+|v_0|)^{-1-d(1+\alpha)/(d+2)} + r_0^{-d}(1+|v_0|)^{-1}\right),
\end{align}
using \eqref{e:QRbound}, \eqref{e:Cbound}, and \eqref{e:L1bound}. Note that we derived \eqref{e:estimate-small-gamma} assuming that $|v_0|\geq 2$. When $|v_0|\leq 2$, the matrix $\bar a_{ij}(z)$ is uniformly elliptic and we can apply Lemma \ref{l:improved} directly to $f$ to obtain \eqref{e:estimate-small-gamma} in this case as well.

\emph{Case 2: $\gamma \in (-2, -1]$.} The argument is the same as in Case 1, but the estimates are quantitatively different as a result of the different bounds on $B(z)$ and $C(z)$ in Lemma \ref{l:T}. The changes are as follows: the radius $R$ of the cylinder $Q_R$ is chosen to be
\[R := K^{(1+\gamma)/d}(r_0/2)(1+|v_0|)^{-(2+\gamma)/2},\]
the bound on $B(z)$ becomes
\[|B(z)|\lesssim  K^{-(1+\gamma)/d} r_0^{1+\gamma}(1+|v_0|)^{(2+\gamma)/2}\leq \Lambda/R, \quad z\in Q_R,\]
and for $C(z)$ we have
\begin{equation*}
|C(z)| \lesssim  \begin{cases} \left(Kr_0^{-d}\right)^{-\gamma/d}(1+|v_0|)^{\gamma}, &\gamma\in \left[\dfrac{-2d}{d+2}, -1\right],\\
\left(Kr_0^{-d}\right)^{-\gamma/d}(1+|v_0|)^{-2-2\gamma/d}, &\gamma\in \left(-2, \dfrac{-2d}{d+2} \right),\end{cases}
\end{equation*}
for $z\in Q_R$. After applying Lemma \ref{l:improved} and \eqref{e:L1bound}, we obtain
\begin{equation*}
f(t_0,x_0,v_0) \leq C\left((K r_0^{-d})^{(d-\gamma)/(d+2)} (1+|v_0|)^{P(d,\alpha,\gamma)} + K^{-(1+\gamma)} r_0^{-d} (1+|v_0|)^{-1}\right),
\end{equation*}
as desired, with $P(d,\alpha,\gamma)$ as in the statement of the lemma.
\end{proof}

We are now in a position to prove our main theorem.

\begin{proof}[Proof of Theorem \ref{t:decay-generation}]
Define
\[K:= \sup\limits_{(0,T_0]\times \R^{2d}} \min\{t^{d/2},1\}f(t,x,v). \]
First, we will show that $K\leq K_*$, where $K_*$ is universal. We can assume $K > 1$. For each $\gamma\in (-2, 0]$, define $p_\gamma: (1,\infty)\to \R$ by
\begin{equation} \label{e:p-equation}
p_\gamma\left(\overline{K}\right) = \begin{cases} C\left(\overline{K}^{(d-\gamma)/(d+2)} + 1\right), & \gamma\in (-1, 0], \\
C\left(\overline{K}^{(d-\gamma)/(d+2)} + (\overline{K})^{-(1+\gamma)}\right), & \gamma\in \left(-2, -1\right], 
\end{cases}
\end{equation}
where $C$ is the appropriate constant from Lemma \ref{l:gain} for each $\gamma$.  Then since $-(1+\gamma)<1$ and $\displaystyle\frac{d-\gamma}{d+2}<1$ for $\gamma>-2$, there is a $K_*>1$ such that
\begin{align*}
K_* &= p_\gamma(K_*), \\
\overline{K}&>p_\gamma(\overline{K}), \quad \mbox{ if }\, \overline{K}>K_*.
\end{align*}
Let $\eps>0$. By the definition of $K$, there exists some $(t_0,x_0,v_0)\in (0,T]\times \R^{2d}$ such that $f(t_0,x_0,v_0)> (K-\eps)\max\{t_0^{-d/2}, 1\}$.  Therefore, Lemma \ref{l:gain} implies that
\[K-\eps \leq p_\gamma(K).\]
Since this is true for all $\eps>0$, we have that $K\leq K_*$.

If $\gamma \in \left[\frac{-2d}{d+2},0\right]$, we apply Lemma \ref{l:gain} with $\alpha=0$ to conclude \eqref{e:decay} with $K_0 =C K_*$. If $\gamma \in \left(-2,\frac{-2d}{d+2}\right)$, Lemma \ref{l:gain} with $\alpha=0$ implies
\[f(t,x,v) \leq C K \left(1+t_0^{-d/2}\right) (1+|v_0|)^{-[d(4+\gamma) + 2 + 2\gamma]/(d+2)},\]
so we can apply Lemma \ref{l:gain} again with $\alpha = [d(4+\gamma) + 2 + 2\gamma]/(d+2)$. We iterate this step, and since for any $\alpha\in (0,1]$, we have $\alpha \leq 1 < d(4+\gamma)/2 + 1 + \gamma$, the gain of decay at each step, $-P(d,\alpha,\gamma) - \alpha$, is bounded away from 0. Therefore, after finitely many steps (with the number of steps depending only on $d$ and $\gamma$), we obtain \eqref{e:decay} for some $K_0$.
\end{proof}

The next result shows that the generating decay in Theorem \ref{t:decay-generation} cannot be improved to polynomial decay with power greater than $d+2$, or to exponential decay. Note that since $\bar b_i = -\partial_j \bar a_{ij}$, for smooth solutions \eqref{e:divergence} may be written equivalently in non-divergence form as
\begin{equation}\label{e:nondivergence}
\partial_t f + v\cdot \nabla_x f= \bar a(t,x,v)D_v^2 f + \bar c(t,x,v)f.
\end{equation}
\begin{theorem}\label{t:polynomial}
Let $\gamma \in [-2,0]$ and $p > d+2$. Assume $f$ solves \eqref{e:divergence} in $[0,T_0]\times \R^{2d}$ with
\begin{equation}\label{e:initial-lower}
f_{in}(x,v) \geq c_0 (1+|v|)^{-p}
\end{equation}
for $v, x\in\R^d$, for some $c_0>0$. Then there exist $c_1>0$ and $\beta>0$ such that
\begin{equation}\label{e:lower-bounds}
f(t,x,v) \geq c_1 e^{-\beta t} (1+|v| ) ^{-p}
\end{equation}
for all $|v|\geq 1, x\in\R^d$, and $t\in [0,T_0]$.
\end{theorem}
\begin{proof}
Let $\eta:\R_+\to\R_+$ be a smooth, decreasing function such that $\eta(r) \equiv 2$ when $r\in[0,\frac 1 2]$ and $\eta(r) = r^{-p}$ when $r\in [1,\infty)$. Note $\eta(r) \approx (1+r)^{-p}$. Let us define $\psi(t,x,v) = e^{-\beta t}\eta(|v|)$ with $\beta$ to be chosen later. Choose an arbitrary $R_0 >1$, and recall from Lemma \ref{l:a} that $\bar a_{ij}\partial_{ij}\psi \geq - C(1+|v|)^{\gamma+2}|D^2\psi|$. (Throughout this proof, $\bar a_{ij}$ and $\bar c$ are defined in terms of $f$.) From our choice of $\eta$, it is clear that $|D^2\psi|/\psi$ is uniformly bounded from above in $\R_+\times \R^d\times\{v: |v|\leq R_0 + 1\}$, so for $\beta\geq \beta_1$ sufficiently large, we have
\[-\partial_t \psi + \bar a_{ij}\partial_{ij}\psi + \bar c \psi \geq \beta \psi - C(1+|v|)^{\gamma+2}|D^2\psi| \geq 0, \quad |v| \leq R_0+1. \]

For $|v|\geq R_0$, we estimate $\bar a_{ij}\partial_{ij}\psi$ more carefully. Since $|v|\geq 1$, we have
\begin{equation*}
\partial_{ij}\psi = \frac {\partial_{rr}\psi}{|v|^2}   v_i v_j + \frac{\partial_r\psi}{|v|} \left( \delta_{ij} - \frac{v_iv_j}{|v|^2}\right) = \left[p(p+1)|v|^{-4} v_i v_j -  p|v|^{-2} \left( \delta_{ij} - \frac{v_iv_j}{|v|^2}\right)\right] e^{-\beta t}|v|^{-p},
\end{equation*}
and Lemma \ref{l:a} implies
\begin{align*}
-\partial_t \psi + \bar a_{ij}\partial_{ij}\psi &\geq  \beta \psi + \left[p(p+1) C_1|v|^{-2+\gamma}  -  pC_2|v|^{\gamma}\right] \psi \geq \left(\beta - C |v|^{\gamma}\right) \psi.
\end{align*}
For $\beta\geq \beta_2$ sufficiently large, the right-hand side is positive for all $|v|\geq R_0$. Since $\bar c(t,x,v)\geq 0$, this implies $\psi(t,x,v) = e^{-\beta t}\eta(|v|)$ with $\beta = \max(\beta_1,\beta_2)$ is a subsolution of $(-\partial_t + \bar a_{ij}\partial_{ij} + \bar c)g = 0$ in the entire domain $\R_+\times \R^{2d}$. By \eqref{e:initial-lower}, there is some $c_1\geq c_0$ so that $f_{in}(x,v) \geq c_1\psi(0,x,v)$ in $\R^{2d}$. Now we can apply the maximum principle (see Appendix \ref{s:A}) to $c_1\psi - f$ to conclude \eqref{e:lower-bounds}.
\end{proof}
\begin{remark}
The bound on the energy $\int_{\R^d} |v|^2 f(t,x,v)\dd v \leq E_0 < \infty$ implies that $f_{in}(x,v)$ cannot be bounded below by  $c_0|v|^{-p}$ with $p\leq d+2$ as $|v|\to\infty$.
\end{remark}

\begin{remark}
In particular, Theorem \textup{\ref{t:polynomial}} tells us that there is no generation of moments when $\gamma \in [-2,0]$.
\end{remark}

\section{Gaussian bounds}\label{s:gaussian}

We show the propagation of Gaussian upper bounds. The first lemma says that a sufficiently slowly decaying Gaussian is a supersolution of the linear Landau equation for large velocities. As above, the coefficients $\bar a_{ij}$ and $\bar c$ in \eqref{e:super} are defined in terms of $f$.
\begin{lemma}\label{l:super}
Let $\gamma \in(-2,0]$. Let $f$ be a bounded function satisfying \eqref{e:M0}, \eqref{e:E0}, and \eqref{e:H0}. Let $\bar a$ and $\bar c$ be given by \eqref{e:a} and \eqref{e:c} respectively. If $\alpha >0$ is sufficiently small, then there exists $R_0>0$ and $C>0$, depending on $d$, $\gamma$, $M_0$, $m_0$, $E_0$, $H_0$ and $\|f\|_{L^\infty}$, such that
\[\phi(v) := e^{-\alpha |v|^2}\]
satisfies
\begin{equation}\label{e:super}
\bar a_{ij} \partial_{ij} \phi + \bar c \phi \leq -C|v|^{\gamma+2}\phi,
\end{equation}
for $|v|\geq R_0$.
\end{lemma}
\begin{proof}
Since $\phi$ is radial, we have
\begin{equation*}
\partial_{ij}\phi = \frac {\partial_{rr}\phi}{|v|^2}   v_i v_j + \frac{\partial_r\phi}{|v|} \left( \delta_{ij} - \frac{v_iv_j}{|v|^2}\right) = \left[\frac {4\alpha^2|v|^2 - 2\alpha}{|v|^2}   v_i v_j -  2\alpha  \left( \delta_{ij} - \frac{v_iv_j}{|v|^2}\right)\right] e^{-\alpha |v|^2},
\end{equation*}
and the bounds \eqref{e:aij-lower} and \eqref{e:aij-upper} imply
\begin{align*}
\bar a_{ij}\partial_{ij}\phi &\leq \left[(4\alpha^2|v|^2 - 2\alpha)C_1 |v|^{\gamma}  -  2\alpha C_2 |v|^{\gamma+2}\right] e^{-\alpha |v|^2}\\
&= \left((4\alpha^2 C_1 - 2\alpha C_2) |v|^{\gamma+2} - 2\alpha C_1 |v|^{\gamma}\right) e^{-\alpha|v|^2}\\
&\leq -C |v|^{\gamma+2} \phi(v),
\end{align*}
for $|v|$ sufficiently large, provided $\alpha < C_2 / (2C_1)$. With Lemma \ref{l:c} (this is the point where $\|f\|_{L^\infty}$ plays a role), this implies
\[\bar a_{ij} \partial_{ij}\phi  + \bar c \phi \leq \left[-C|v|^{\gamma+2} + C|v|^{-2-2\gamma/d}\right]\phi(v).\]
For $-2<\gamma \leq 0$, the first term on the right-hand side will dominate for large $|v|$, since $\gamma+2 > 0 > -2-2\gamma/d$.
\end{proof}

Theorem \ref{t:decay-generation} gives us an upper bound for a solution $f$ to the Landau equation which is useful away from $t=0$. If the initial data $f(0,x,v)$ is a bounded function, we can improve our upper bound for small values of $t$ using the upper bound for $f(0,x,v)$. That is the purpose of the next lemma.
\begin{lemma}\label{l:maximum-principle}
Let $f: [0,T_0]\times \R^{2d}\to \R$ be a solution of the Landau equation \eqref{e:divergence} for some $\gamma\in (-2, 0]$, and suppose that $g: [0,T_0]\times \R^{2d}\to \R$ is bounded from above and a subsolution to the equation
\begin{equation}\label{e:subequation}
\partial_t g(t,x,v) + v\cdot \nabla_x g(t,x,v) \leq \bar a_{ij}(t,x,v) \partial_{ij} g(t,x,v) + \bar c(t,x,v) g(t,x,v),
\end{equation}
where $\bar a_{ij}$ and $\bar c$ are defined in terms of $f$ as in \eqref{e:a} and \eqref{e:c}. Let $\kappa(t)$ be defined by
\begin{equation}\label{e:kexponent}
\kappa(t) = \left\{\begin{array}{cl} \frac{\beta}{1+\gamma/2} t^{1+\gamma/2}, & 0\leq t \leq 1 \\ \frac{\beta}{1+\gamma/2} + \beta(t-1), & t\geq 1 \end{array} \right. ,
\end{equation}
where $\beta>0$ depends only on $d$, $\gamma$, $m_0$, $M_0$, $E_0$, and $H_0$.
Then
\begin{equation}\label{e:max}
\sup\limits_{[0,T_0]\times \R^d} e^{-\kappa(t)}g_+(t,x,v) = \sup\limits_{\R^{2d}} g_+(0,x,v). \end{equation}
\end{lemma}

\begin{proof}

By Theorem \ref{t:decay-generation}, we have that $f(t,x,v)\leq K_0t^{-d/2}$ for $0<t<1$.  Hence by Lemma \ref{l:c}, we have that $\bar c(t,x,v) \lesssim t^{\gamma/2}$. Since $\gamma > -2$, for some universal $\beta>0$, $\kappa(t)$ satisfies $\bar c(t,x,v)\leq \kappa'(t)$ for all $t>0$.  Thus $\tilde{g}(t,x,v) = e^{-\kappa(t)}g(t,x,v)$ satisfies
\begin{align*}
\partial_t \tilde{g}(t,x,v) + v\cdot \nabla_x\tilde g(t,x,v) &\leq \bar a_{ij}(t,x,v)\partial_{ij}\tilde{g}(t,x,v) + (\bar c(t,x,v)-\kappa'(t))\tilde{g}(t,x,v)\\
 &\leq \bar a_{ij}(t,x,v)\partial_{ij}\tilde{g}(t,x,v).
\end{align*}
We apply Lemma \ref{l:maximum-unbounded} from the Appendix to $\tilde g(t,x,v) - \displaystyle\sup_{\R^{2d}} g(0,x,v)$ to conclude \eqref{e:max}.
\end{proof}

\begin{proof}[Proof of Theorem \ref{t:propagation-gaussian}]


Applying Lemma \ref{l:maximum-principle} with $g=f$ and $t \in [0,1]$ and Theorem \ref{t:decay-generation} for $t>1$, we have that there is some constant $C_2$ (depending on $C_0$, $d$, $\gamma$, $M_0$, $m_0$, $E_0$, and $H_0$) so that $f(t,x,v) \leq C_2$ for all $t \geq 0$, $x \in \R^d$ and $v \in \R^d$.

Let $\phi(v) := e^{-\alpha|v|^2}$. From Lemma \ref{l:super}, we have that there is a $C$, depending on $C_2$, $d$, $\gamma$, $M_0$, $m_0$, $E_0$, and $H_0$, such that
\[\sup\limits_{(0,T_0]\times \R^d \times \R^d} \bar a_{ij}\partial_{ij} \phi + \bar c \phi \leq C \phi.\]
Thus $C_0 e^{Ct} \phi(v)$ is a supersolution of the equation and $f(t,x,v) \leq C_0 e^{Ct} \phi(v)$ for all $t >0$, $x \in \R^d$ and $v \in \R^d$.

This upper bound is good for small values of $t$. We see that there is some time $t_0 > 0$ so that $C_0 e^{Ct_0} \phi(v) > C_2$ for $|v| < R_0$. Here $C_2$ is the upper bound for $f$ mentioned above and $R_0$ is the radius from Lemma \ref{l:super}. Thus, the function
\[ g(t,x,v) :=\left[ f(t_0+t,x,v) - C_0 e^{Ct_0} \phi(v) \right]_+\]
is a supersolution of
\[ g_t + v \cdot \nabla_x g \leq \bar a_{ij} \partial_{ij} g + \bar c g. \]
Applying the maximum principle (Lemma \ref{l:maximum-unbounded}), we have that $g \leq 0$ for all $t>0$, so $f(t,x,v) \leq C_0 e^{t_0 C} \phi(v)$ for all $t>t_0$, and we conclude the proof.
\end{proof}

By combining Theorem \ref{t:propagation-gaussian} with the local H\"older estimates proved in \cite{golse2016} or \cite{wang2011ultraparabolic}, we derive a global H\"older estimate for solutions of \eqref{e:divergence} under the assumption that $f_{in}(x,v) \leq C_0 e^{-\alpha |v|^2}$. The following local estimate is essentially the same as Theorem 2 of \cite{golse2016}:
\begin{theorem}\label{t:local-holder}
Let $f$ be a weak solution of
\begin{equation}\label{e:FP2}
\partial_t f + v\cdot \nabla_x f = \nabla_v \cdot(A\nabla_v f) + B \cdot\nabla_v f + s
\end{equation}
in $Q_1$, with $\lambda I \leq A \leq \Lambda I$, $|B|\leq \Lambda$, and $s\in L^\infty(Q_1)$. Then $f$ is H\"older continuous with respect to $(t,x,v)$ in $Q_{1/2}$, and
\[\frac{|f(z_1) - f(z_1)|}{|t_1-t_2|^{\beta/2}+|x_1-x_2|^{\beta/3}+|v_1+v_2|^\beta} \leq C(\|f\|_{L^2(Q_1)} + \|s\|_{L^\infty(Q_1)}),\]
for all $z_1,z_2\in Q_{1/2}$, where $\beta$ and $C$ depend on $d$, $\lambda$, and $\Lambda$.
\end{theorem}


To state our theorem as a global H\"older estimate, we will need an appropriate notion of distance in $\R \times \R^d \times \R^d$ which is invariant by Galilean transformations. A natural choice is the following
\[ d_P(z_1,z_2) := \min \{ r : \exists z \in \R \times \R^d \times \R^d : z_1 \in Q_r(z) \text{ and } z_2 \in Q_r(z) \}.\]
We can easily estimate the value of $d_P(z_1,z_2)$ by the simpler formula
\[ d_P(z_1,z_2) \approx |t_1-t_2|^{1/2} + |x_1 - x_2 - (t_1-t_2)(v_1+v_2)/2|^{1/3} + |v_1-v_2|.\]

It turns out that we need to deform this distance using the transformation $\mathcal T_z$ described in Lemma \ref{l:T}. We define
\[ d_L(z_1,z_2) := \min \{ |v|^{1+\gamma/2} r: z \in \R \times \R^d \times \R^d : \mathcal T_z^{-1} z_1 \in Q_{r} \text{ and } \mathcal T_z^{-1} z_2 \in Q_{r} \}.\]
(Here, we make the convention that $\mathcal T_{z} = \mathcal S_z$ when $|v|< 2$.) An explicit expression for $d_L(z_1,z_2)$ is messy. It involves the affine transformation $\mathcal T$ which is anisotropic and affects both the $x$ and $v$ variables. In the case that we compare two points with identical values of $t$ and $x$, it is straightforward to check that when $ d_L((t,x,v_1),(t,x,v_2))<1$, then $d_L$ is equivalent to the metric introduced by Gressman-Strain \cite{gressman2011boltzmann} in their study of the Boltzmann equation.



\begin{theorem}\label{t:holder}
Under the assumptions of Theorem \textup{\ref{t:propagation-gaussian}}, there exist $C>0$ and $\beta \in (0,1)$ depending on $C_0$, $\alpha$, $d$, $\gamma$, $m_0$, $M_0$, $E_0$, and $H_0$, such that for any $z_1,z_2\in [0,T_0]\times \R^{2d}$, one has
\begin{align*}
|f(z_1) - f(z_2)| &\leq C\left(e^{-\alpha|v_1|^2} + e^{-\alpha|v_2|^2}\right)\min\left\{ 1, \left(1+ t_1^{-\beta/2}+t_2^{-\beta/2}\right)
d_L(z_1,z_2)^\beta\right\}.
\end{align*}
\end{theorem}

\begin{proof}
If $|v_1|\leq 2$ or $|v_2| \leq 2$, the result follows by applying Theorem \ref{t:local-holder} directly to $f$, noting that $1\lesssim e^{-\alpha|v_1|^2}+e^{-\alpha|v_2|^2}$. So, we can assume that $|v_1| > 2$ and $|v_2|>2$.

Let $\bar z = (\bar t, \bar x, \bar v)$ be the point achieving the minimum in the definition of $d_L(z_1,z_2)$. Thus $\tilde z_1 := \mathcal T_{\bar z}^{-1} z_1 \in Q_\delta$ and $\tilde z_1 := \mathcal T_{\bar z}^{-1} z_1 \in Q_\delta$, where $\delta = |\bar v|^{-1-\gamma/2}d_L(z_1,z_2)$.

Let $r := \min\left(t_1^{1/2}, t_2^{1/2},(1+|\bar v|)^{-1-\gamma/2}\right)$. If $\delta \geq r/2$, then we simply estimate $|f(z_1) - f(z_2)| \leq C_1(e^{-\alpha|v_1|^2}+e^{-\alpha|v_2|^2})$ from Theorem \ref{t:propagation-gaussian}. We need to concentrate on the case $\delta < r/2$.

Let us consider the function $f_T$ as in Lemma \ref{l:T}, with base point $\bar z$. By our choice of $r$, $f_T$ satifies an equation of the form \eqref{e:FP2} in $Q_r$, and since Theorem \ref{t:propagation-gaussian} gives us a bound on $\|f_T\|_{L^\infty}$, we have that $A$ is uniformly elliptic (with constants independent of $\bar z$), $|B|\lesssim |\bar v|^{1+\gamma/2}$, and $|s|=|C(z)f_T| \lesssim |f_T|$. Defining $\tilde f_T(t,x,v) := f_T(r^2t,r^3x,rv)$, we see that $\tilde f_T$ satisfies another equation of the form \eqref{e:FP2} with the new $|B|$ bounded independently of $|\bar v|$. Moreover, the points $(r^{-2} \tilde t_1, r^{-3} \tilde x_1, r^{-1} \tilde v_1)$ and $(r^{-2} \tilde t_2, r^{-3} \tilde x_2, r^{-1} \tilde v_2)$ belong to $Q_{r^{-1} \delta} \subset Q_{1/2}$. Therefore, we can apply Theorem \ref{t:local-holder} to $\tilde f_T$ in $Q_1$ to obtain
\begin{align*}
\frac{|f(z_1) - f(z_2)|}{r^{-\beta} d_L(z_1,z_2)^\beta}|\bar v|^{\beta(1+\gamma/2)} &= \frac{|f_T(\tilde z_1) - f_T(\tilde z_2)|}{r^{-\beta} \delta^\beta} \\
&= \frac{|\tilde f_T(r^{-2} \tilde t_1, r^{-3} \tilde x_1, r^{-1} \tilde v_1) - \tilde f_T(r^{-2} \tilde t_2, r^{-3} \tilde x_2, r^{-1} \tilde v_2)|}{r^{-\beta} \delta^\beta}, \\
&\lesssim \|\tilde f_T\|_{L^1(Q_1)} + \|\tilde f_T\|_{L^\infty(Q_1)} \lesssim  \sup_{v\in Q_{r/2}}e^{-\alpha|\bar v+Tv|^2}  \lesssim e^{-\alpha|\bar v|^2}.
\end{align*}
We have used Theorem \ref{t:propagation-gaussian} to estimate the $L^\infty$ norm of $\tilde f_T$ in $Q_1$. Rewriting this estimate, we obtain
\begin{align*}
 |f(z_1) - f(z_2)| &\lesssim r^{-\beta} |\bar v|^{-\beta(1+\gamma/2)} d_L(z_1,z_2)^\beta e^{-\alpha|\bar v|^2}\\
  &\lesssim \left(1+t_1^{-\beta/2}+t_2^{-\beta/2}\right) d_L(z_1,z_2)^\beta \left(e^{-\alpha|\bar v_1|^2} + e^{-\alpha|\bar v_2|^2}\right) .
\end{align*}
\end{proof}

\appendix

\section{Maximum principle for weak solutions to kinetic Fokker-Planck equations}\label{s:A}

In this appendix, we give a proof of the maximum principle in a form that is convenient for our purposes.

The following proposition is perhaps a classical result. We prove it here, since we could not find any easy reference and also for completeness. The result is for equations on a bounded domain with general coefficients (not necessarily defined by integrals as above).

\begin{proposition}\label{p:maximum-bounded}
Let $Q = [0,T_0]\times \Omega$, where $\Omega \subset \R^{2d}$ is a bounded domain, and assume that $g$ is a subsolution of the equation
\begin{equation}\label{e:weak-inequality}
\partial_t g + v\cdot \nabla_x g \leq \nabla_v\cdot [a(t,x,v) \nabla_v g] + b(t,x,v)\cdot \nabla_v g + c(t,x,v) g,
\end{equation}
in the weak sense in $Q$, where $a$ is uniformly elliptic in $Q$ with constants $\lambda$ and $\Lambda$, and $b$ and $c$ are uniformly bounded in $Q$.

If $g\leq 0$ on the parabolic boundary of $Q$, then $g\leq 0$ in $Q$.
\end{proposition}

\begin{proof}
Choosing the test function $\phi = g_+$, the weak formulation of \eqref{e:weak-inequality} gives
\begin{align*}
\int_Q g_+ \left(\partial_t g +  v\cdot \nabla_x g\right) \dd x \dd v \dd t &\leq \int_Q \left(-a \nabla_v g \nabla_v g_+ - g_+b \cdot \nabla g + c g_+^2\right)\dd x \dd v \dd t,
\end{align*}
or
\begin{align*}
\int_Q\frac 1 2\frac d {dt} (g_+)^2 \dd x \dd v \dd t &\leq \int_Q \left(-\lambda |\nabla_v g_+|^2 - b g \nabla g_+ + c g_+^2\right)\dd x \dd v \dd t\\
&\leq \left(\frac{\|b\|_{L^\infty}}{4\lambda} + \|c\|_{L^\infty}\right) \int_Q g_+^2\dd x \dd v \dd t,
\end{align*}
by Young's inequality. We apply Gronwall's Lemma to $\int_{\Omega} (g_+)^2\dd x \dd v$ on $[0, T_0]$ to conclude $g_+ \equiv 0$ in $Q$.
\end{proof}

Next, we derive a maximum principle on the whole space for subsolutions of a Landau-type equation without a zeroth-order term:
\begin{lemma}\label{l:maximum-unbounded}
Let $g$ be a bounded function on $[0,T_0]\times \R^{2d}$ that satisfies
\begin{equation}\label{e:weak-subsolution-2}
\partial_t g + v\cdot \nabla_x g \leq \bar a(t,x,v) D^2_v g,
\end{equation}
in the weak sense. Here, $\bar a(t,x,v)$ is defined as in \eqref{e:a} in terms of a function $f$ satisfying \eqref{e:M0}, \eqref{e:E0}, and \eqref{e:H0}. If $g(0,x,v) \leq 0$ in $\R^{2d}$, then $g(t,x,v) \leq 0$ in $[0,T_0]\times \R^{2d}$.
\end{lemma}
\begin{proof}
By the bounds on $\bar a$ given in Lemma \ref{l:a}, we have
\[ \bar a_{ij}\partial_{ij} (1+|v|) \leq C_1(1+|v|)^{1+\gamma}, \]
for some constant $C_1$, and thus $\phi_1(t,v) := e^{C_1t} (1+|v|)$ satisfies
 \[ \partial_t \phi_1(t,v) \geq \bar a_{ij}(t,x,v) \partial_{ij} \phi_1(t,v).\]
Let $\eps_1>0$ be a small constant. Since $g$ is bounded, there is $R(\eps_1)>0$ such that $g - \eps_1 \phi_1<0$ whenever $|v|\geq R(\eps_1)$. Let $R_1> R(\eps_1)$, and choosing $C_2>0$ large enough depending on $R_1$, we can define $\phi_2(t,x) := (1+|x|) e^{C_2 t}$, and we have
\[ \partial_t \phi_2 + v\cdot \nabla_x \phi_2 \geq 0,\]
whenever $|v|<R_1$. Finally, for $\eps_2>0$ arbitrary, we define
\[\tilde g(t,x,v) := [g(t,x,v) - \eps_1 \phi_1(t,v) - \eps_2 \phi_2(t,x)]_+.\]
It is clear that $\tilde g$ is a subsolution as in \eqref{e:weak-inequality} with $c\equiv 0$, whenever $|v|<R_1$. For $R(\eps_2)$ sufficiently large, we have that $g - \eps_1\phi_1 - \eps_2\phi_2 <0$ for $|x|\geq R(\eps_2)$ or $|v|\geq R(\eps_1)$. Then for any $R_2>R(\eps_2)$, we have that $\tilde g = 0$ on the parabolic boundary of $[0,T_0]\times B_{R_2}\times B_{R_1}$, so Proposition \ref{p:maximum-bounded} applied to $\tilde g$ gives
\[ g - \eps_1\phi_1 - \eps_2 \phi_2 \leq 0,\quad  |v|<R_1,\, |x|< R_2.\]
Take $R_2\to\infty$ and $\eps_2\to 0$ to conclude
\[ g - \eps_1\phi_1 \leq 0, \quad |v|<R_1.\]
Take $R_1\to\infty$ and $\eps_1\to 0$, and the proof is complete.
\end{proof}
\bibliographystyle{abbrv}
\bibliography{landau}

\end{document}